\documentclass[12pt]{article}
\usepackage{amssymb,amsmath,amsthm}
\usepackage{epsfig}
\usepackage{psfrag}

\newcommand{\bt}{\begin{Theorem}}
\newcommand{\et}{\end{Theorem}}
\newcommand{\bi}{\beign{itemize}}
\newcommand{\ei}{\end{itemize}}
\newcommand{\bea}{\begin{eqnarray}}
\newcommand{\eea}{\end{eqnarray}}
\newtheorem{Theorem}{\sc Theorem}[section]
\newtheorem{Lemma}[Theorem]{\sc Lemma}
\newtheorem{Proposition}[Theorem]{\sc Proposition}
\newtheorem{Corollary}[Theorem]{\sc Corollary}
\newtheorem{Definition}[Theorem]{\sc Definition}

\newtheorem{Remark}[Theorem]{\sc Remark}

\newcommand{\be}{\begin{equation}}
\newcommand{\ee}{\end{equation}}

\def\C{\mathbb C}
\newcommand{\CA}{\mathcal{A}}
\newcommand{\CB}{\mathcal{B}}
\newcommand{\CC}{\mathcal{C}}

\def\be{\begin{equation}}
\def\ee{\end{equation}}
\def\bt{\begin{Theorem}}
\def\et{\end{Theorem}}
\def\bi{\begin{itemize}}
\def\ei{\end{itemize}}
\def\bea{\begin{eqnarray}}
\def\eea{\end{eqnarray}}
\def\beast{\begin{eqnarray*}}
\def\eeast{\end{eqnarray*}}
\def\ben{\begin{enumerate}}
\def\een{\end{enumerate}}

\def\rar{\rightarrow}
\def\Rar{\Rightarrow}

\def\Lra{{\Leftrightarrow}}

\begin{document}

\title{On a tensor-analogue of the Schur product}
\author{K. Sumesh and V.S. Sunder\\
Institute of Mathematical Sciences\\
Chennai 600113\\
INDIA\\
{\small email: sumeshkpl@gmail.com, sunder@imsc.res.in}}

\maketitle

\begin{abstract}
We consider the {\em tensorial Schur product} $R \circ^\otimes S = [r_{ij} \otimes s_{ij}]$ for $R \in M_n(\CA), S\in M_n(\CB),$ with $\CA, \CB ~\mbox{unital}~ C^*$-algebras, verify that such a `tensorial Schur product' of positive operators is again positive, and then use this fact to prove (an  apparently  marginally more general version of) the classical result of Choi  that a linear map $\phi:M_n \rar M_d$  is completely positive if and only if $[\phi(E_{ij})] \in M_n(M_d)^+$, where of course $\{E_{ij}:1 \leq i,j \leq n\}$ denotes the usual system of matrix units in $M_n (:= M_n(\C))$. We also discuss some other corollaries of the main result.
\end{abstract}

\section{The result}

\noindent We start with some notation: (We assume, for convenience, that all our $C^*$-algebras are unital.) We denote an element of a matrix algebra by capital letters, such as $R$, and denote its entries by either $[R]_{ij}$ or the corresponding lower case letter $r_{ij}$. This is primarily because $[R^*]_{ij} = (r_{ji})^* \neq [R^*]_{ji}!$ 

\begin{Definition}\label{phi-A-tsp}
\ben
    \item 
         If $\CA, \CB$ are $C^*$-algebras, and $\phi : M_n \rar \CB$ is a positive map, define  $\phi_{\CA}:A \otimes M_n \rar \CA \otimes_{alg} \CB$ (where $\CA \otimes_{alg} \CB$ denotes the algebraic tensor product of $\CA$ and $\CB$) by $\phi_{\CA}  = id_{\CA} \otimes \phi$.
    \item
         If $A = [a_{ij}] \in M_n(\CA), B=[b_{ij}] \in M_n(\CB)$, define $A \circ^\otimes B = [a_{ij}\otimes b_{ij}] \in M_n(\CA \otimes_{alg}\CB)$.
\een 
\end{Definition}

For later use, we isolate a lemma, whose elementary verification we omit.

\begin{Lemma}\label{jlem}
The map $\pi: M_n(\CA) \otimes M_k \rar M_{nk}(\CA)$ defined by
\be \label{ipjq}
[\pi(R \otimes C)]_{i\alpha,j\beta} = c_{\alpha\beta}r_{ij}
\ee
is a $C^*$-algebra isomorphism for any $C^*$-algebra $\CA$; in the sequel, we shall simply use this $\pi$ to make the identification $M_n(\CA) \otimes M_k = M_{nk}(\CA)$. In particular, $\CA \otimes M_k = M_{k}(\CA)$.
\end{Lemma} 

\begin{Remark}\label{lr} There is clearly a right version of the above Lemma: i.e., $M_k \otimes M_n(\CA) = M_{kn}(\CA)$. \end{Remark}

\begin{Proposition}\label{p1} 
\be \label{stp}
    R\in M_n(\CA)^+, S \in M_n(\CB)^+ \Rar R \circ^\otimes S \in M_n(\CC)^+,
\ee 
where $\CC$ denotes - here and in the rest of this short note - any $C^*$-algebra containing $\CA \otimes_{alg}\CB$.
In particular,
\be \label{sumcor}
    \sum_{i,j=1}^n r_{ij}\otimes s_{ij} \in M_n(\CC)^+.
\ee
\end{Proposition}

\begin{proof} 
To deduce eqn. (\ref{sumcor}) from eqn. (\ref{stp}), we  let $1_n \in M_{n \times 1}(\CC)^+$ be the $n \times 1$ column-vector with all entries equal to $1_\CA \otimes 1_\CB$, and note that 
\[  \sum_{i,j=1}^n r_{ij}\otimes s_{ij}  = 1_n^* ( R \circ^\otimes S) 1_n.\]

Now for the slightly less immediate eqn. (\ref{stp}). By assumption, $R \otimes S \in M_{n^2}(\CC)^+$. 

In the sequel all the variables $i,j,k,l,p,q$ will range over the set $\{1,2,...,n\}$ and we shall simply write $\sum_k$ for $\sum_{k=1}^n$. 

Now define $V \in M_{n \times n^2}(\CC)$ by
\[[V]_{i,pq} = \delta_{pi} \delta_{qi} (1_A \otimes 1_B).\]
Then,
\beast
   [V(R\otimes S) V^*]_{ij} &=& \sum_{p,q,k,l} [V]_{i,pq} [R\otimes S]_{pq,kl} [V^*]_{kl,j}\\
                            &=& [R\otimes S]_{ii,jj}\\
                            &=& r_{ij}\otimes s_{ij}
\eeast
and so,
\[V(R\otimes S) V^* = R \circ^\otimes S.\]
The proof of the Proposition is complete.
\end{proof}

\begin{Remark}\label{tspinalg}
Note that the proof shows that $R\circ^\otimes S \in M_n(\CA \otimes_{alg} \CB)$.
\end{Remark}

The classical result of Choi alluded to in the abstract is the equivalence $2. \Lra 3. $ in the following Corollary, for the case $\CB = M_d$ (see \cite{Choi}).

\begin{Corollary}\label{choi}
The following conditions on a linear map $\phi : M_n \rar \CB$ are equivalent:
\begin{enumerate}
     \item
          For any $C^*$-algebra $\CA$, the map $\phi_\CA (:= id_\CA \otimes \phi) : \CA \otimes M_n \rar \CC$ is a positive map for any $C^*$-algebra $\CC$ as in Proposition \ref{p1}.     
    \item
          The map $\phi$ is CP.\footnote{For an explanation of terms like CP (= completely positive) and operator system, the reader may consult \cite{Pis},  for instance.}
         \item
          $[\phi(E_{ij})] \in M_n(\CB)^+$.
\end{enumerate}
\end{Corollary}

\begin{proof}
We only prove the non-trivial implication $3. \Rar 1.$ if $R \in (\CA \otimes M_n)^+ = M_n(\CA)^+$, and if  $R = [r_{ij}]$, then
\beast
    \phi_\CA (R) &=& (id_\CA \otimes \phi) (\sum_{ij} r_{ij}\otimes E_{ij})\\
                 &=& \sum_{ij} r_{ij} \otimes \phi(E_{ij})\\
                 &\in& M_n(\CC)^+,
\eeast
by eqn. (\ref{sumcor}).
\end{proof}

\begin{Corollary}\label{LR}
Let $R\in M_n(\CA)^+$. Then the map $M_n(\CB) \ni S \stackrel{L_R}{\rar} R \circ^\otimes S \in M_n(\CC)$ is CP. In particular $R \in M_n^+ \Rar M_n \ni S \rar R \circ S \in M_n$ is also CP.
\end{Corollary}

\begin{proof} 
To avoid confusion, we use Greek letters $\alpha, \beta$ etc., to denote elements of $\{1,2,\cdots,k\}$ and English letters $i,j$ etc. to denote elements of $\{1,2,\cdots,n\}$.  Suppose  $[\hat{S}] \in M_{kn}(\CB)^+ = M_k(M_n(\CB))^+$ is given by $[\hat{S}]_{\alpha i,\beta j} := [S_{\alpha,\beta}]_{i,j}$ (see Lemma \ref{jlem}), where of course $S_{\alpha\beta} \in M_n(\CB) \forall  \alpha,\beta $). Let  $J_k \in M_k$ be (the all 1 matrix) given by $[J_k]_{\alpha\beta} = 1 ~\forall \alpha, \beta$. Then we see that
$J_k \geq 0$ (in fact $J_k/k$ is a projection) and so,

\beast
[L_R(S_{\alpha\beta})] &=& [[r_{ij} \otimes [S_{\alpha \beta}]_{ij}]]\\
&=& [[[J_k]_{\alpha\beta} r_{ij} \otimes [S_{\alpha \beta}]_{ij}]]\\
&=& [[J_k \otimes R]_{\alpha i,\beta j} \otimes [\hat{S}]_{\alpha i,\beta j}]~~(\mbox{see Remark }\ref{lr})\\
&=&  (J_k \otimes R) \circ^\otimes \hat{S}\\
& \geq & 0, 
\eeast
by Proposition \ref{p1} applied with $R,n,S$ there replaced by $J_k \otimes R, kn, \hat{S}$, since $[J_k \otimes R]_{\alpha i,\beta j} = r_{ij} ~\forall \alpha,\beta$. The second statement of the Corollary is just the specialisation of the first statement to $\CA = \CB = \C$.
\end{proof}

\begin{Remark}\label{Lr}
The special case $n=1$ of Corollary \ref{LR} perhaps merits singling out: If $r\in \CA^+$, then the map $\CB \ni s \stackrel{L_r}{\rar} r \otimes s \in \CC$ is CP.
\end{Remark}

\begin{Remark}\label{right}
It should be clear that there is a `right' version of all the `left' statements discussed above.
\end{Remark}

The proofs suggest that these results might well admit formulations in the language of operator systems; however, we suspect that such `generalisations' will follow from nuclearity of $M_n$ and the flexibility in the choice of $\CC$ in our formulation, in view of the Choi-Effros theorem (see \cite{ChoEff}).


\begin{thebibliography}{5}
   \bibitem {Choi} M. D. Choi, {\em Completely positive linear maps on complex matrices}, Linear Algebra and Appl. {\bf 10} (1975), 285--290.
   \bibitem {ChoEff} M. D. Choi\ and\ E. G. Effros, {\em Injectivity and operator spaces}, J.  of Functional Analysis {\bf 24} (1977), no.~2, 156--209. 
   \bibitem {Pis} G. Pisier,  {\em Introduction to Operator Space Theory}. LMS Lecture Note Series 294, Cambridge University Press, 2003.  
 \end{thebibliography}
\end{document}